\documentclass[11pt]{amsart}

\usepackage{amsthm,amsfonts,amssymb,amsmath}

\usepackage{graphicx,color}
\usepackage{amscd}
\usepackage{setspace}
\usepackage{CJK}
\usepackage[margin=1in]{geometry}

\usepackage[english]{babel}

\usepackage{comment}

\usepackage{ytableau}
\usepackage{youngtab}

\numberwithin{equation}{section}

\newtheorem{thm}{Theorem}
\newtheorem{lemma}[thm]{Lemma}
\newtheorem{proposition}[thm]{Proposition}
\newtheorem{corollary}[thm]{Corollary}

\newtheorem*{MainThm}{Theorem \ref{thm:main1}}

\newtheorem*{MainThm2}{Theorem \ref{thm:main2}}

\theoremstyle{definition}
\newtheorem{example}[thm]{Example}
\newtheorem{definition}[thm]{Definition}
\newtheorem{remark}[thm]{Remark}

\numberwithin{thm}{section}

\newcommand{\PP}{\mathbb{P}}
\newcommand{\CC}{\mathbb{C}}

\renewcommand{\L}{\mathcal{L}}
\newcommand{\Gr}{\mathrm{Gr}}

\newcommand{\RED}{*(white!80!red) \color{black}}
\newcommand{\BLUE}{\color{blue}}
\newcommand{\GRAY}{*(white!80!black)}
\newcommand{\TrSSYT}{\mathrm{TrSSYT}}

\title[A Generalized RSK for Enumerating Linear Series on $n$-pointed Curves]{A Generalized RSK for Enumerating \\ Linear Series on $n$-pointed Curves}

\author{Maria Gillespie}
\thanks{Partially supported by NSF DMS award number 2054391.} 
\address{Department of Mathematics, Colorado State University, Fort Collins, CO, USA} \email{maria.gillespie@colostate.edu }

\author{Andrew Reimer-Berg}
\address{Department of Mathematics, Colorado State University, Fort Collins, CO, USA
}
\email{andrew.reimer-berg@colostate.edu}
\date{\today}

\begin{document}

\maketitle

\begin{abstract}
    We give a combinatorial proof of a recent geometric result of Farkas and Lian on linear series on curves with prescribed incidence conditions.  The result states that the expected number of degree-$d$ morphisms from a general genus $g$, $n$-marked curve $C$ to $\PP^r$, sending the marked points on $C$ to specified general points in $\PP^r$, is equal to $(r+1)^g$ for sufficiently large $d$.  This computation may be rephrased as an intersection problem on Grassmannians, which has a natural combinatorial interpretation in terms of Young tableaux by the classical Littlewood-Richardson rule.   We give a bijection, generalizing the well-known RSK correspondence, between the tableaux in question and the $(r+1)$-ary sequences of length $g$, and we explore our bijection's combinatorial properties.  
    
    We also apply similar methods to give a combinatorial interpretation and proof of the fact that, in the modified setting in which $r=1$ and several marked points map to the same point in $\PP^1$, the number of morphisms is still $2^g$ for sufficiently large $d$. 
\end{abstract}

\section{Introduction}
In a recent paper \cite{FarkasLian}, Farkas and Lian provide enumerative formulas for the number of maps from a curve $C$ to a complex projective space $\PP^r$ with specified incidence conditions.  In particular, let $C$ be a general curve of genus $g$, and let $x_1,\ldots,x_n$ be distinct general points on $C$.   Also choose distinct general points $y_1,\ldots,y_n$ in $\PP^r$.  Then we write $L_{g,r,d}$ for the number of degree $d$ morphisms $$f:C\to \PP^r$$ for which $f(x_i)=y_i$ for all $i=1,2,\ldots,n$, which is finite precisely when $nr=dr+d+r-rg$. 

For sufficiently large $d$ and setting $n=(dr+d+r-rg)/r$, it was shown in \cite{FarkasLian} that \begin{equation}
    L_{g,r,d}=(r+1)^g.
\end{equation} In addition, $L_{g,r,d}$ is equal to a certain intersection of Schubert cycles in the Grassmannian.  The latter formula has the following combinatorial interpretation, as we will show in Section \ref{sec:L-interp}.  

\begin{definition}
Define an \textbf{$L$-tableau with parameters $(g,r,d)$} to be a way of filling the boxes of an $(r+1)\times (d-r)$ grid with $rg$ ``red'' integers and $(d-r)(r+1)-rg$ ``blue'' integers such that:
\begin{itemize}
    \item The red integers are left-and-bottom justified, and weakly increase up columns and strictly increase across rows.  They consist of the numbers $1,2,\ldots,g$ each occurring exactly $r$ times.
    \item The blue integers, which are necessarily right-and-top justified, are strictly increasing up columns and weakly increasing across rows.  Their values are from $\{0,1,\ldots,r\}$.
\end{itemize}
\end{definition}

\begin{example}\label{ex:Lgrd}
  The following is an $L$-tableau with parameters $(4,3,9)$.  We write the ``red'' numbers as black font with a red shaded background for clarity.
  \begin{center}
      \begin{ytableau}
 \RED 2 & \RED 4 & \BLUE 1 & \BLUE 3 & \BLUE 3 & \BLUE 3 \\
 \RED 1 & \RED 3 & \RED 4 & \BLUE 1 & \BLUE 2 & \BLUE 2 \\
 \RED 1 & \RED 2 & \RED 3 & \BLUE 0 & \BLUE 1 & \BLUE 1 \\
 \RED 1 & \RED 2 & \RED 3 & \RED 4 & \BLUE 0 & \BLUE 0 \\
      \end{ytableau}
  \end{center}
\end{example}

The work in \cite{FarkasLian} shows that the number of $L$-tableaux with parameters $(g,r,d)$ is equal to $(r+1)^g$  whenever $d\ge rg+r$ (and $r|d$, so that $n$ is an integer) via geometric methods, and asks for a combinatorial proof.  Our first main result resolves this open problem by finding a combinatorial proof, of the following stronger result. 

\begin{thm}\label{thm:main1}
  The number of $L$-tableaux with parameters $(g,r,d)$ is $(r+1)^g$ whenever $d\ge g+r$.
\end{thm}

 Note that in this purely combinatorial setting we only require $d\ge g+r$ rather than $d\ge rg+r$, and $d$ is not necessarily divisible by $r$.

Our methods generalize the \textbf{RSK algorithm} (see Section \ref{sec:RSK}).  In particular, in the case $r=1$, the pair of red and blue tableaux correspond directly under the RSK bijection to the binary sequences of length $g$, as we will show in Section \ref{sec:r=1}.  For $r>1$ we introduce an intermediate bijection (see Definition \ref{def:varphi}) to reduce to the RSK correspondence once again. 

\begin{remark}
In the case $r=1$, there are several known proofs of the fact that $L_{g,1,d}=2^g$ for sufficiently large $d$, including via \textit{scattering amplitudes} for $d=g+1$ \cite{Tevelev2020} and by establishing recursions using the boundary geometry of the moduli space of \textit{Hurwitz covers} \cite{CPS}.  Neither of these proofs were combinatorial in nature, though the recursions arising in the latter paper by Cela-Pandharipande-Schmitt \cite{CPS} are related to Dyck paths and other Catalan objects for small $d$.
\end{remark}

Our second main result enumerates a related set of maps.  In particular, set $r=1$, fix an integer $k$ with $1\le k\le \min(n,d)$, and consider choices of points $y_1,\ldots,y_n\in \PP^1$ such that $$y_1=y_2=\cdots=y_k.$$  Write $L'_{g,d,k}$ for the number, assuming that $n=(dr+d+r-rg)/r=2d+1-g$, of maps $f:C\to \PP^1$ such that $f(x_i)=y_i$ for all $i$.  We use a similar interpretation in terms of a family of Young tableaux that we call \textbf{$L'$-tableaux}, starting from an intersection theoretic formula in Grassmannians for $L'_{g,d,k}$ given by Farkas and Lian \cite{FarkasLian}, to enumerate these maps for sufficiently high degree curves.

\begin{thm}\label{thm:main2}
  If $d\ge g+k$, we have $L'_{g,d,k}=2^g$.
\end{thm}

Notice that the formula $2^g$ coincides with $(r+1)^g$ at $r=1$.  Indeed, it is remarked in \cite{FarkasLian} that the simple formula $L'_{g,d,k}=2^g$ may be explained geometrically in the same way that the formula $L_{g,r,d}=(r+1)^g$ does.  Further details and a more general formula for maps with arbitrary ramification profiles were given by Cela and Lian \cite{CelaLian}.  However, the formula from \cite{FarkasLian} coming from Grassmannians is stated as a \textit{difference} of two sums of Schubert class intersection products (see Section \ref{sec:Lprime-interp}), and there was previously no stated enumerative combinatorial interpretation of the latter formula, or for that matter a direct explanation for why the difference should be positive.   We provide such an interpretation in order to prove Theorem \ref{thm:main2}.  

This also raises the question of whether there is a natural generalization of $L'$-tableaux that enables one to get a handle on the $r>1$ setting with $y_1=\cdots=y_k$.

\subsection{Outline}
In Section \ref{sec:background} below we recall the necessary background results and notation from tableaux theory and intersection theory on Grassmannians.  
In Section \ref{sec:grd} we translate the geometric formulas for $L_{g,r,d}$ into a Young tableaux enumeration problem and prove Theorem \ref{thm:main1}.  We also show that our bijection reduces to ordinary RSK in the case $r=1$ (Section \ref{sec:r=1}), and allows us to recover a classical theorem of Castelnuovo in the case $d=r+\frac{rg}{r+1}$ (Section \ref{sec:Castelnuovo}).  

In Section \ref{sec:gdk} we define $L'$-tableaux for $L'_{g,d,k}$ and prove Theorem \ref{thm:main2}.  Finally, in Section \ref{sec:combinatorics}, we more fully explore the combinatorial properties of our constructions.

\subsection{Acknowledgments}

We thank Carl Lian for several enlightening conversations pertaining to this work, and thank Jake Levinson for bringing our attention to the open combinatorial problem.  We also thank Renzo Cavalieri, Alexander Hulpke, and Mark Shoemaker for their helpful feedback.

\section{Background}\label{sec:background}

We briefly recall several known facts about Young tableaux, linear series, and Grassmannians.

\subsection{The RSK correspondence on words}\label{sec:RSK}

A \textbf{partition} of $n$ is a nonincreasing sequence of positive integers $\lambda=(\lambda_1,\ldots,\lambda_k)$ for which $|\lambda|:=\sum_i \lambda_i=n$.  The \textbf{Young diagram} of $\lambda$ is the left- and bottom-justified grid of unit squares in the first quadrant (called \textit{boxes}) for which there are $\lambda_i$ boxes in the $i$-th row from the bottom for all $i$\footnote{Here we are using the so-called `French' convention for drawing Young diagrams.}. A \textbf{semistandard Young tableau}, or \textbf{SSYT}, of shape $\lambda$ is a way of filling the boxes of the Young diagram with nonnegative integers such that the rows are weakly increasing from left to right and the columns are strictly increasing from bottom to top.  An SSYT is \textbf{standard}, written SYT, if the entries are $1,2,\ldots,n$ each used exactly once.  The \textbf{shape} of a Young tableau is the underlying (unlabeled) Young diagram.  Its \textbf{content} is the tuple $(m_1,m_2,\ldots)$ where $m_i$ is the number of times $i$ appears in the tableau.  For instance, an SYT has content $(1^n)=(1,1,\ldots,1)$.

The RSK correspondence, in its most general form, is a bijection between lexicographically-sorted two-line arrays and pairs of semistandard Young tableaux of the same shape (see \cite{Fulton} for an excellent overview).  It is constructed via an algorithmic insertion procedure starting from the two-line array.
In the case that the bottom row of the two-line array consists of the numbers $1,2,3,\ldots,n$ in order, the top row may be any sequence, and we obtain the following special case of the RSK correspondence. 

\begin{proposition}[RSK for words]\label{prop:RSK}
  Let $A(r,n)=\{0,1,2,3,\ldots,r\}^n$ denote the set of all length-$n$ sequences with entries from $\{0,1,2,\ldots,r\}$.  Let $B(r,n)$ be the set of all pairs $(P,Q)$ such that $P$ is a semistandard Young tableau with letters in $\{0,1,2,\ldots,r\}$, $Q$ is a standard Young tableau, and $P$ and $Q$ have the same shape of size $n$.

  There is an explicit bijection, called the RSK correspondence, from $A(r,n)$ to $B(r,n)$ for all $r,n$.
\end{proposition}

We will refer to the length-$n$ sequences of letters from $\{0,1,\ldots,r\}$ as \textbf{$(r+1)$-ary sequences}, generalizing the notion of a \textbf{binary sequence} from $\{0,1\}$.  

We do not require the full definition of the RSK bijection explicitly here, and we refer the reader to \cite{Fulton} for details.  We will use the following property that follows from the definition of semistandard Young tableaux. 

\begin{remark}
  Since $P$ has letters in $\{0,1,2,3,\ldots,r\}$, and columns are strictly increasing,
  any pair $(P,Q)$ in $B(r,n)$ can have height at most $r+1$.
\end{remark}

\begin{example}
  The $4$-ary sequence $0,2,1,1,0,3,0,0,1$ in $A(3,9)$ corresponds under RSK to the following pair of tableaux in $B(3,9)$: $$\begin{ytableau}
  2 \\ 1 & 1 & 3 \\ 0 & 0 & 0 & 0 & 1 \end{ytableau}\,\,,\,\,
  \begin{ytableau}
  5 \\ 3 & 7 & 8 \\ 1 & 2 & 4 & 6 & 9 \end{ytableau}$$ 
\end{example}

\subsection{Linear series and calculations in the Grassmannian}

The definitions of $L_{g,r,d}$ and $L'_{g,r,d}$ may be made more rigorous via the theory of linear series on curves.  (See \cite{EisenbudHarris} for an introduction to this topic.) A  \textbf{linear series} of type $\mathfrak{g}^r_d$ on a smooth curve $C$ can be thought of as a $r$-dimensional linear family of sets of $d$ points on $C$.  More formally, this data is encoded by a pair $(\L, V)$ where $\L$ is a line bundle of degree $d$ and $V\subseteq H^0(C,\L)$ is a dimension $r+1$ space of sections, which in turn corresponds to a map $\phi_{(\L,V)}$ from $C$ to $\PP^r$.  

There has been much study of enumeration of linear series with prescribed \textit{ramification} conditions at specified points, including the work of Eisenbud and Harris \cite{EisenbudHarris83}, Osserman \cite{Osserman}, Chan and Pflueger \cite{ChanPflueger1}, Chan, L\'{o}pez Mart\'{i}n, Pflueger, and Teixidor i Bigas \cite{CLPT}, Larson, Larson, and Vogt \cite{LLV}, and others.  Many of these results relied on the combinatorial tools of Young tableaux, symmetric function theory, and other aspects of algebraic combinatorics, and in \cite{ChanPflueger2} led to new results in algebraic combinatorics as well.

Here we consider the related problem of enumerating linear series with prescribed \textit{incidence} conditions at specified marked points.   Write $G^r_d(C)$ to denote the moduli space of $\mathfrak{g}^r_d$'s on $C$, and let $x_1,\ldots,x_n$ be general marked points on $C$.  Then the values $L_{g,r,d}$ may alternatively be defined as the degree of the evaluation map $\mathrm{ev}_{(x_1,\ldots,x_n)}$ from $G^r_d(C)$ to the moduli space $P^n_r$ of $n$ points in $\PP^r$, given by evaluating the maps $\phi_{(\L,V)}$ at the points $x_1,\ldots,x_n$. 

In \cite{FarkasLian}, a degeneration argument on linear series, starting with a reduction to genus $0$, is used to reduce the problem of computing the degree of this evaluation map to a standard intersection problem in Schubert calculus.  For genus $0$ curves $C$, $H^0(C,\mathcal{L})$ in general has dimension $d+1$ as a complex vector space for sufficiently high $d$.  Therefore, the $r+1$-dimensional subspaces $V\subseteq H^0(C,\mathcal{L})$ sweep out a copy of the Grassmannian $\mathrm{Gr}(r+1,d+1)$.

The Grassmannian $\mathrm{Gr}(r+1,d+1)$, defined as the moduli space of $r+1$-dimensional subspaces of $\CC^{d+1}$, has a well-known \textbf{Schubert decomposition} (with respect to a given flag) into Schubert varieties $X_\lambda$.  (See \cite[Ch.\ 9]{Fulton} or \cite{Gillespie} for background on Schubert calculus.) Here, $\lambda$ ranges over all partitions $\lambda=(\lambda_1, \lambda_2,\ldots, \lambda_k)$ for which which $$\lambda_1\le d+1 \hspace{1cm}\text{and}\hspace{1cm} k\le r+1.$$  In other words, the Young diagram of $\lambda$ fits inside an $(r+1)\times (d+1)$ grid, as shown in Figure \ref{fig:young-diagram}.

\begin{figure}
    \centering
    \begin{ytableau}
     \empty & \empty & \empty & \empty & \empty & \empty\\
     \GRAY & \empty & \empty & \empty & \empty& \empty \\
     \GRAY & \GRAY & \empty & \empty & \empty& \empty \\
     \GRAY & \GRAY & \GRAY & \GRAY & \GRAY& \empty \\  
     \GRAY & \GRAY & \GRAY & \GRAY & \GRAY& \empty\\     
    \end{ytableau}
    \caption{The Young diagram (shaded) of the partition $(5,5,2,1)$, drawn inside a $5\times 6$ grid.  This corresponds to the Schubert class $\sigma_{(5,5,2,1)}$ in $A^\bullet(\Gr(5,11))$.}
    \label{fig:young-diagram}
\end{figure}

The Schubert varieties give rise to a basis of \textbf{Schubert classes} $\sigma_\lambda:=[X_\lambda]$ of its Chow ring $A^\bullet(\Gr(r+1,d+1))$.  With respect to this basis, it is shown in \cite{FarkasLian} that whenever either $d\ge rg+r$, $d=r+\frac{rg}{r+1}$, or $r=1$, we have 
\begin{equation}\label{eq:Lgrd}
    L_{g,r,d}=\int_{\Gr(r+1,d+1)} \sigma_{1^r}^g\cdot \left[ \sum_{\alpha_0+\cdots+\alpha_r=(r+1)(d-r)-rg}\left(\prod_{i=0}^r \sigma_{\alpha_i}\right)\right].
\end{equation}
Here the notation $\sigma_{1^r}$ is shorthand for $\sigma_{(1,1,1,\ldots,1)}$ where the tuple $(1,1,\ldots,1)$ has length $r$, and $\sigma_{\alpha_i}$ is shorthand for $\sigma_{(\alpha_i)}$.  The integral indicates that the sum of products of Schubert cycles in question expands in the Schubert basis as a constant multiple of $$\sigma_{(d+1)^r}:=\sigma_{(d+1,d+1,\ldots,d+1)},$$ and the integral is defined to be this constant coefficient.

The corresponding result from \cite{FarkasLian} for $L'_{g,d,k}$ is that for $k+g\le 2d+1$ and $2\le k\le d$, we have
\begin{equation}\label{eq:Lprime}
    L'_{g,d,k}=\int_{\Gr(2,d+1)} \sigma_1^g \sigma_{k-1}\left(\sum_{i+j=2d-g-k-1}\sigma_i\sigma_j\right)-\int_{\Gr(2,d)}\sigma_1^g \sigma_{k-2}\left(\sum_{i+j=2d-g-k-2} \sigma_i\sigma_j\right).
\end{equation}

\subsection{The iterated Pieri rule}\label{sec:LR-rule}

The intersections of Schubert cycles on the Grassmannian may be calculated via symmetric function theory, as products of Schubert classes correspond to products of \textit{Schur functions}.  Indeed, let $\{s_\lambda\}$ be the classical Schur function basis of the ring of symmetric functions, where $\lambda$ ranges over all partitions (see \cite{Fulton} or \cite[Ch.\ 7]{Stanley}).  Then the integral in equation \eqref{eq:Lgrd} is equal to the coefficient of $s_{(d+1)^{(r+1)}}$ in the expansion \begin{equation}\label{eq:schur}
s_{1^r}^g\cdot \left[ \sum_{\alpha_0+\cdots+\alpha_r=(r+1)(d-r)-rg}\left(\prod_{i=0}^r s_{\alpha_i}\right)\right]
\end{equation}
The coefficients of products of Schur functions expressed in the Schur basis are called \textbf{Littlewood-Richardson coefficients}.  In particular we can write $$s_\lambda\cdot s_\mu=\sum c^{\nu}_{\lambda\mu} s_\nu$$ where the Littlewood-Richardson coefficients $c^{\nu}_{\lambda\mu}$ are all nonnegative integers. 

Many combinatorial formulas for the Littlewood-Richardson coefficients are known (see \cite{Fulton} for several interpretations via Young tableaux alone).  In our setting we will only need to focus on the cases when one of $\lambda$ or $\mu$ is either a horizontal row or vertical column, which are often called the horizontal and vertical \textbf{Pieri rules}.  We recall these rules here.

\begin{definition}
    Suppose $\lambda$ and $\mu$ are partitions for which the Young diagram of $\mu$ is a subset of the diagram of $\lambda$. The \textbf{skew shape} $\lambda/\mu$ is the set of boxes that are in $\lambda$ but not in $\mu$.
    
    A skew shape is a \textbf{horizontal strip} if no two of its boxes are in the same column, and it is a \textbf{vertical strip} if no two of its boxes are in the same row.  
\end{definition}

Examples of horizontal and vertical strips are shown in Figure \ref{fig:skew}.  We also say the strip \textbf{extends} the inner partition $\mu$.  The following rules are well-known (see, for instance, \cite[pp.\ 24--25]{Fulton}.)

\begin{figure}
    \centering
    \begin{ytableau}
     \empty & \GRAY & \GRAY \\
     \empty & \empty & \empty & \GRAY \\
     \empty & \empty & \empty & \empty \\
     \empty & \empty & \empty & \empty &\GRAY & \GRAY & \GRAY
    \end{ytableau}\hspace{2cm}
    \begin{ytableau}
      \empty & \empty & \GRAY \\
     \empty  & \empty & \GRAY \\
     \empty  & \empty & \empty \\
     \empty  & \empty & \empty &\GRAY
    \end{ytableau}
    \caption{At left, the horizontal strip $(6,4,4,3)/(4,4,3,1)$. At right, the vertical strip $(4,3,3,3)/(3,2,2,2)$.  Each is drawn as a set of shaded boxes.}
    \label{fig:skew}
\end{figure}

\begin{proposition}[Pieri rules]\label{prop:pieri}
  For $\mu=(\alpha)$ a single-row partition, the Littlewood-Richardson coefficient $c^\nu_{\lambda\mu}=c^\nu_{\lambda,(\alpha)}$ is equal to $1$ if $\nu/\lambda$ is a horizontal strip, and $0$ otherwise.
  
  For $\mu=(1^r)$ a single-column partition, the Littlewood-Richardson coefficient $c^\nu_{\lambda\mu}=c^\nu_{\lambda,(1^r)}$ is equal to $1$ if $\nu/\lambda$ is a vertical strip, and $0$ otherwise.
\end{proposition}

Proposition \ref{prop:pieri} gives us a rule for multiplying any Schur function by either $s_{(1^r)}$ or $s_\alpha$, and expanding the result again as a sum of Schur functions.  In particular, $$s_\lambda \cdot s_{(\alpha)}=\sum_{\nu/\lambda\in \mathrm{Horz(\alpha)}} s_{\nu} \hspace{1cm} s_{\lambda}\cdot s_{(1^r)}=\sum_{\nu/\lambda\in \mathrm{Vert(r)}} s_{\nu}$$ where $\mathrm{Horz}(\alpha)$ and $\mathrm{Vert}(r)$ are the sets of all horiztonal strips of size $\alpha$ and vertical strips of size $r$ respectively. 

We can iterate to give a rule for any product of row or column Schur functions.  For instance, multiplying both sides of the left hand equation above by $s_{(1^r)}$ gives $$s_\lambda \cdot s_{(\alpha)}\cdot s_{(1^r)}=\sum_{\nu/\lambda\in \mathrm{Horz}(\alpha)}s_\nu \cdot s_{(1^r)}=\sum_{\nu/\lambda\in \mathrm{Horz}(\alpha)}\left(\sum_{\rho/\nu\in \mathrm{Vert}(r)}s_\rho\right).$$ Interchanging the order of summation, we see that the number of times the Schur function $s_\rho$ appears in this expansion is equal to the number of ways to extend $\lambda$ by a horizontal strip of size $\alpha$ and then by a vertical strip of size $r$ in order to fill shape $\rho$.  This observation may be generalized as follows.  

\begin{corollary}[Iterated Pieri rule]\label{cor:iPieri}
  Let $\mu^{(1)},\mu^{(2)},\ldots,\mu^{(k)}$ be partitions, each of which is either a horizontal row or a vertical column.  Then $$s_{\mu^{(1)}}\cdot s_{\mu^{(2)}}\cdot \cdots \cdot s_{\mu^{(k)}}=\sum c^{\nu}_{\mu^{(1)}\cdots\mu^{(k)}}s_{\nu}$$ where $c^{\nu}_{\mu^{(1)}\cdots\mu^{(k)}}$ is equal to the number of ways to extend $\mu^{(1)}$ by horizontal or vertical strips (as indicated by each $\mu^{(i)}$) of sizes $|\mu^{(2)}|,\ldots,|\mu^{(k)}|$ such that the total resulting shape is $\nu$.
\end{corollary}

In order to keep track of the horizontal and vertical strips, we will label the squares of the strip corresponding to $\mu^{(i)}$ by $i$ for each $i$.  This results in a tableau-like object that enumerates the generalized Pieri coefficients $c^\nu_{\mu^{(1)}\cdots\mu^{(k)}}$ above.

\begin{example}
  The coefficient $c^{(4,2,1)}_{(2),(1,1),(3)}$ is equal to $2$, because there are two ways to fill the boxes of shape $(4,2,1)$ with a horizontal strip of two $1$'s, a vertical strip of two $2$'s extending it, and a horizontal strip of three $3$'s that extends the shape again:  $$\begin{ytableau}
   2 \\
   2 & 3 \\
   1 & 1 & 3 & 3 
  \end{ytableau}\hspace{2cm}\begin{ytableau}
   3 \\
   2 & 3 \\
   1 & 1 & 2 & 3 
  \end{ytableau}$$
\end{example}

In the special case when we have all horizontal strips, we recover the well-known notion of a \textbf{semistandard Young tableau}, or \textbf{SSYT}: a filling of the boxes of a (possibly skew) Young diagram with numbers such that the rows are weakly increasing from left to right and the columns are strictly increasing from bottom to top.  For the case of all vertical strips, we say a \textbf{transposed semistandard Young tableau} is a filling of a Young diagram with strictly increasing \textit{rows} and weakly increasing \textit{columns}.  We similarly obtain transposed SSYT's in the case of all vertical strips.  We summarize these observations in the following remark.

\begin{remark}
Notice that $c^\nu_{(\alpha_1),\ldots,(\alpha_k)}$ is equal to the number of semistandard Young tableaux of shape $\nu$ with exactly $\alpha_i$ $i$'s for each $i$.  Similarly, $c^\nu_{(1^{r_1}),\ldots,(1^{r_k})}$ is the number of transposed semistandard Young tableaux of shape $\nu$ with exactly $r_i$ $i$'s for each $i$.
\end{remark}

\section{$L$-tableaux and enumeration by $(r+1)^g$}\label{sec:grd}

We now have the tools to show that the $L$-tableaux with parameters $(g,r,d)$ do indeed enumerate the integrals $L_{g,r,d}$ in the Grassmannian, starting from equation \eqref{eq:Lgrd}.  We will then show that the $L$-tableaux are enumerated by $(r+1)^g$.

\subsection{The $L$-tableaux}\label{sec:L-interp} Corollary \ref{cor:iPieri}, combined with the fact that the integral in equation \eqref{eq:Lgrd} is the coefficient of $s_{(d-r)^{r+1}}$ in the corresponding product \eqref{eq:schur}, shows that $$L_{g,r,d}=\sum_{\alpha_0+\cdots+\alpha_r+rg=(r+1)(d-r)} c^{(d-r)^{r+1}}_{(1^r),(1^r),\ldots,(1^r),(\alpha_0),\ldots,(\alpha_r)}$$ where the subscripts on the coefficient contain $g$ copies of $(1^r)$.  This summation is therefore the number of ways to form a transposed SSYT using each of the numbers $1,2,\ldots,g$ exactly $r$ times, and then extend it to fill the rest of the $(r+1)\times (d-r)$ grid with a semistandard Young tableau using the numbers $0,1,\ldots,r$ in some varying amounts $\alpha_0,\ldots,\alpha_r$ each.

This precisely matches the definition of $L$-tableaux given in the introduction, which we restate in terms of our new notation here.

\begin{definition}\label{def:Lgrdtab}
An \textbf{$L$-tableau with parameters $(g,r,d)$} is a  way of filling an $(r+1)\times (d-r)$ rectangular grid with:
\begin{itemize}
    \item (The `red' tableau.) A transposed SSYT having exactly $r$ copies of each of the numbers $1,2,\ldots,g$.  That is, its content is $(r^g)=(r,r,r,\ldots,r)$.
    \item (The `blue' tableau.) A semistandard Young tableau on the remaining skew shape of boxes, with values from $\{0,1,\ldots,r\}$.
\end{itemize}
\end{definition}

\begin{remark}
The preprint \cite{FarkasLian} mistakenly uses an ordinary (not transposed) SSYT for the red tableau; a correction will appear in a later version of their work \cite{Lian2021}.
\end{remark}

See Example \ref{ex:Lgrd} for an example.  Our discussion thus far, starting from Equation \eqref{eq:Lgrd}, has shown:

\begin{proposition}
  The number of $L$-tableau with parameters $(g,r,d)$ is equal to $L_{g,r,d}$ whenever either $d\ge rg+r$, $d=r+\frac{rg}{r+1}$, or $r=1$.
\end{proposition}

We now show that we can ``truncate'' by removing some of the right-hand columns of the grid to reduce to a simpler case.

\begin{lemma}[Truncation]\label{lem:chop}
  For any $g,r,d$ with $d\ge g+r$, the number of $L$-tableaux with parameters $(g,r,d)$ is equal to the number of $L$-tableaux with parameters $(g,r,g+r)$.  
\end{lemma}

\begin{proof}
  Suppose $d\ge g+r$.  Notice that, since it is transposed semistandard, the red tableau has width at most $g$, since its bottom row is strictly increasing from left to right and uses only the numbers $1,2,\ldots,g$.  Therefore, any column to the right of the $g$-th column is filled entirely with blue numbers, which strictly increase up the columns using the numbers $0,1,2,\ldots,r$, necessarily exactly once since the columns have height $r+1$. 
  
  It follows that there is only one way to fill each of the columns to the right of column $g$, and these columns therefore do not contribute to the enumeration.  We therefore may remove the last $d-r-g$ columns and find that the number of $L$-tableaux with parameters $(g,r,d)$ equals the number with parameters $(g,r,g+r)$.
\end{proof}

Lemma \ref{lem:chop} tells us that in order to understand $L_{g,r,d}$ for $d\ge g+r$, it suffices to study the case $d=g+r$.  We will restrict to this case throughout the remainder of this section. 

\begin{remark}
When $d=g+r$, the rectangle containing the $L$-tableaux is size $(r+1)g=rg+g$.  The red tableau has size $rg$ and so the blue tableau has size $g$.
\end{remark}

\subsection{Enumeration by $(r+1)^g$}

In this section we prove Theorem \ref{thm:main1}.  We first define the following sets of tableaux.

\begin{definition}
    Let $\TrSSYT(g,r)$ be the set of all transposed SSYT's of content $(r^g)=(r,r,\ldots,r)$ and height $\le r+1$.
    \label{def:red}
\end{definition}

Note that $\TrSSYT(g,r)$ is the set of all possible `red' tableaux in Definition \ref{def:Lgrdtab}.  We will refer to them as \textbf{red tableaux} throughout this section.

\begin{definition}
Define a \textbf{$180^\circ$-rotated SYT} to be the result of rotating a standard Young tableaux $180^\circ$ in the plane.  We write $\mathrm{SYT}^{180^\circ}(g,r)$ for the set of all $180^\circ$-rotated SYT of size $g$ and height $\le r+1$.  

We informally call such a tableau a \textbf{purple tableau}, as it will be used as an intermediate object relating the red and blue tableaux of Definition \ref{def:Lgrdtab}.
\end{definition}

\begin{example}
    Below is an example of a purple tableau in $\mathrm{SYT}^{180}(7,3)$.
    
  $$\color{violet}
    \begin{ytableau}
      7 & 3 & 1 \\
      \none & 6 & 2 \\
      \none & \none & 4 \\
      \none & \none & 5
    \end{ytableau}$$
\end{example}

Given a red tableau, note that each number $1,\ldots,g$ occurs once in every row except one. Relatedly, a purple tableau in the position of the blue tableau will have each number $1,\ldots,g$ in exactly one row.  This leads us to define a bijection between the two as follows.

\begin{definition}[Red to purple bijection] \label{def:varphi}  Let $R\in \TrSSYT(g,r)$ be a red tableau.  We define a $180^\circ$-rotated tableau $\varphi(R)$ in the upper right corner of a rectangle by the following iterative process.  We add boxes labeled $1,2,\ldots,g$ in order, where on the $i$th step we place a box labeled $i$ as far to the right as possible in the unique row that does \textbf{not} contain an $i$ in $R$.
\end{definition}

\begin{example}
If $R$ is the tableau at left below, $\varphi(R)$ is shown at right below.
  $$\begin{ytableau}
   \RED 2 & \RED 4 & \RED 5 & \RED 6\\
   \RED 1 & \RED 3 & \RED 4 & \RED 5 & \RED 7\\
   \RED 1 & \RED 2 & \RED 3 & \RED 5 & \RED 6 & \RED 7\\
   \RED 1 & \RED 2 & \RED 3 & \RED 4 & \RED 6 & \RED 7
  \end{ytableau}
  \hspace{1cm} 
  {\color{violet}
  \begin{ytableau}
   7 & 3 & 1 \\
   \none & 6 & 2 \\
   \none & \none & 4 \\
   \none & \none & 5
  \end{ytableau}}$$
  \label{ex:redpurp}
\end{example}

We now show $\varphi$ is a bijection.  We note that a generalized version of this map was shown to be a bijection in \cite{REINER19981} (see also Section \ref{sec:general}), but we include a direct proof here for the special case that we are considering, for the reader's convenience.

\begin{lemma}\label{lem:RedToPurple}
	The map $\varphi$ is a bijection from $\TrSSYT(g,r)$ to $\mathrm{SYT}^{180^\circ}(g,r)$ for all $g,r$.  Moreover, for any $R\in \TrSSYT(g,r)$, the shapes of $R$ and $\varphi(R)$ are complementary in an $(r+1)\times g$ rectangle.
\end{lemma}

\begin{proof}
We show both statements by induction on $g$. For $g=1$, the tableau $R$ must be a column of $1$'s of height $r$, and $\varphi(R)$ is a single $1$ in the top row. They are clearly complementary in an $(r+1)\times 1$ rectangle (column).  

For $g>1$, assume the statements hold for $g-1$ and let $R\in \TrSSYT(g,r)$.  Consider the tableau $R'$ formed by removing the vertical strip of $g$'s from $R$.  Let $T'=\varphi(R')$.  Then $T'$ and $R'$ are complementary in an $(r+1)\times (g-1)$ rectangle by the inductive hypothesis.  

By shifting $T'$ one unit to the right, we form an empty vertical strip $V$ of size $r+1$ between the two tableaux.  Then all $r$ of the $g$'s in $R$ must lie in this strip, and in fact the one remaining square $x$ must be at the top of some column of $V$.  Then, the square $x$ is precisely the one that we label $g$ to form $T=\varphi(R)$ starting from $T'$, by Definition \ref{def:varphi}. Since the entries immediately above and to the right of $x$ will have entries smaller than $g$ (or $x$ is on the right hand or top edge of the rectangle), this construction forms a $180^\circ$-rotated SYT $T$, so $\varphi$ is well-defined and the resulting pair $(R,\varphi(R))$ is complementary. 

Finally, note that by the induction hypothesis, $\varphi$ is a bijection for $g-1$, and the possible squares we can add to $T'$ to form $g$ are precisely in bijection with the possible sub-strips of $g$'s of the vertical strip $V$ that we may add to $R'$ to form $R$.  Thus $\varphi$ is a bijection for size $g$ as well.
\end{proof}

We now make precise the notion of a ``blue tableau'' (see Definition \ref{def:Lgrdtab}).

\begin{definition}
    Define a \textbf{$180^\circ$-rotated semistandard tableau}, or \textbf{blue tableau} (with parameters $r,g$), to be a filling of a $180^\circ$-rotated Young diagram of size $g$ with numbers $0,1,2,\ldots,r$ such that the rows are weakly increasing from left to right and columns are strictly increasing from bottom to top.
\end{definition}

\begin{example}
    Below is an example of a blue tableau. It has the same shape as the purple tableau above in Example \ref{ex:redpurp}.
  $$\begin{ytableau}
   \BLUE 0 & \BLUE 2 & \BLUE 3 \\
   \none & \BLUE 1 & \BLUE 2 \\
   \none & \none & \BLUE 1 \\
   \none & \none & \BLUE 0
  \end{ytableau}$$
\end{example}

\begin{lemma}\label{lem:BluePurpleRSK}
	The pairs of blue and purple tableaux of the same shape correspond to $(r+1)$-ary sequences of length $g$ bijectively, via inverting the entries of the blue tableau (that is, replacing each entry $i$ by $r-i$), rotating both $180^\circ$, and applying RSK.
\end{lemma}

\begin{proof}
	Given a blue tableau $S$ and purple tableau $T$ of the same shape,	rotate both by 180 degrees to form $S^{180}$ and $T^{180}$. Then, $T^{180}$ is a standard tableau, which we call $Q$. We also form a semistandard tableau $P$ out of $S^{180}$ by inverting its entries;
	that is, we replace each $i$ in $S^{180}$ with $r-i$ in $P$.  
	
	Then, $(P,Q)$ is a pair in $B(r,g)$ (see Proposition \ref{prop:RSK}), so by the RSK bijection on words, we have a bijection between these pairs $(P,Q)$ and $A(r,g)$, which is precisely the set of $(r+1)$-ary sequences of length $g$.
\end{proof}

\begin{example}
  Consider the pair of blue and purple tableaux below.
  $$
  \begin{ytableau}
   \BLUE 0 & \BLUE 2 & \BLUE 3 \\
   \none & \BLUE 1 & \BLUE 2 \\
   \none & \none & \BLUE 1 \\
   \none & \none & \BLUE 0
  \end{ytableau}
  \hspace{1cm} 
  {\color{violet}
  \begin{ytableau}
   7 & 3 & 1 \\
   \none & 6 & 2 \\
   \none & \none & 4 \\
   \none & \none & 5
  \end{ytableau}}$$
 The corresponding pair $(P,Q)$ is as follows:
  $$
  \begin{ytableau}
   3 \\
   2 \\
   1 & 2 \\
   0 & 1 & 3
  \end{ytableau}
  \hspace{1cm}
  \begin{ytableau}
   5 \\
   4 \\
   2 & 6 \\
   1 & 3 & 7
  \end{ytableau}$$
  Then, via RSK, this pair corresponds to the $(r+1)$-ary sequence $3,2,2,1,0,1,3$ where $r=3$.
\end{example}

We finally have the tools to produce a bijection between $L$-tableaux and $(r+1)$-ary sequences.

\begin{proposition} \label{prop:bijection}
 The $L$-tableaux with parameters $(g,r,g+r)$ are in bijection with the $(r+1)$-ary sequences of length $g$ (with letters from the alphabet $\{0,1,2,\ldots,r\}$).
\end{proposition}

\begin{proof}
	Each such $L$-tableaux consists of a red tableau and a blue tableau. The bijection follows from combining Lemma \ref{lem:RedToPurple} with Lemma \ref{lem:BluePurpleRSK}, which provide bijections between red tableaux with purple tableaux, and between pairs of blue and purple tableaux with $(r+1)$-ary sequences of length $g$, respectively.
\end{proof}

\begin{example}
    Below is an $L$-tableau with parameters $(7,3,10)$. From our previous examples, we see that it corresponds to the $(r+1)$-ary sequence $3,2,2,1,0,1,3$.

    $$
    \begin{ytableau}
        \RED 2 & \RED 4 & \RED 5 & \RED 6 & \BLUE 0 & \BLUE 2 & \BLUE 3 \\
        \RED 1 & \RED 3 & \RED 4 & \RED 5 & \RED 7 & \BLUE 1 & \BLUE 2 \\
        \RED 1 & \RED 2 & \RED 3 & \RED 5 & \RED 6 & \RED 7 & \BLUE 1 \\
        \RED 1 & \RED 2 & \RED 3 & \RED 5 & \RED 6 & \RED 7 & \BLUE 0 \\
    \end{ytableau}
    $$
\end{example}

There are $(r+1)^g$ sequences of length $g$ in the alphabet $0,1,2,\ldots,r$.  Combining Proposition \ref{prop:bijection} with truncation (Lemma \ref{lem:chop}), we get as a corollary Theorem \ref{thm:main1}.

\begin{MainThm}
The number of $L$-tableaux with parameters $(g,r,d)$ is $(r+1)^g$ for all $d\ge r+g$.
\end{MainThm}

We now analyze two special cases of our construction.

\subsection{The case $r=1$}\label{sec:r=1}
We claim that at $r=1$, the composition of bijections discussed above reduces to ordinary RSK.   Indeed, in this case, the red tableau is simply a standard Young tableau on the numbers $1,2,\ldots,g$ of height $2$. The blue tableau consists of $0$'s and $1$'s, and when rotated 180 degrees is the same shape as the red tableau (and is semistandard after interchanging $0$'s and $1$'s).  More precisely, we have the following.

\begin{proposition}
  When $r=1$, the bijection $\varphi$ (Definition \ref{def:varphi}) from red to purple tableaux reduces to $180^\circ$ rotation.
\end{proposition}

We illustrate this with an example.  Consider the $L$-tableau with parameters $(5,1,6)$ below:
$$\begin{ytableau}
 \RED 3 & \BLUE 0 & \BLUE 1 & \BLUE 1 & \BLUE 1 \\
 \RED 1 & \RED 2 & \RED 4 & \RED 5 & \BLUE 0
\end{ytableau}$$
The bijection $\varphi$ applied to the red tableau above gives:
$$\begin{ytableau}
 \RED 3   \\
 \RED 1 & \RED 2 & \RED 4 & \RED 5
\end{ytableau}
\hspace{1cm}
{\color{violet}
\begin{ytableau}
 5 & 4 & 2 & 1 \\
 \none & \none & \none & 3
\end{ytableau}}$$
and combining the purple tableau on the right with the blue (after $180^\circ$ rotation and switching $1$'s and $0$'s gives the pair:
$$\begin{ytableau}
 \BLUE 1   \\
 \BLUE 0 & \BLUE 0 & \BLUE 0 & \BLUE 1
\end{ytableau}
\hspace{1cm}
{\color{violet}
\begin{ytableau}
 3 \\
 1 & 2 & 4 & 5
\end{ytableau}}
$$
which corresponds under RSK to the binary sequence $0,1,0,0,1$.

\subsection{The Castelnuovo case}\label{sec:Castelnuovo}
In \cite{FarkasLian}, the authors consider another special case, when $d=r+\frac{rg}{r+1}$, and show that their formula \begin{equation*}
    L_{g,r,d}=\int_{\Gr(r+1,d+1)} \sigma_{1^r}^g\cdot \left[ \sum_{\alpha_0+\cdots+\alpha_r=(r+1)(d-r)-rg}\left(\prod_{i=0}^r \sigma_{\alpha_i}\right)\right]
\end{equation*} holds in this case as well.  In fact, by construction we have $(r+1)(d-r)-rg=0$ and so the integral above reduces to the simple product $\sigma_{1^r}^g$ in the cohomology ring of the Grassmannian.   

Due to work of Castelnuovo \cite{Castelnuovo} and Griffiths and Harris \cite{Griffiths1980}, it is known that this quantity equals 
$$g!\cdot \frac{1!\cdot 2!\cdots r!}{s!\cdot (s+1)!\cdot \cdot \cdots \cdot (s+r)!}$$
where $s=\frac{g}{r+1}$ (which must be an integer since $\frac{rg}{r+1}=d-r$ is an integer and $r$ and $r+1$ are relatively prime).  We observe here how this may be enumerated directly via a variant of $L$-tableaux, using our `red to purple' bijection $\varphi$ of Definition \ref{def:varphi}.

Indeed, the integral $\int_{\Gr(r+1,d+1)}\sigma_{1^r}^g$  is the coefficient of $s_{((d-r)^{r+1})}$ in the product $s_{(1^r)}^g$, which by Corollary \ref{cor:iPieri} is the number of transposed SSYT's having shape a $(r+1)\times (d-r)$ rectangle and exactly $r$ of each letter $1,2,\ldots,g$.  Note that $d-r=\frac{rg}{r+1}$ so the entire rectangle has $$(r+1)\cdot \frac{rg}{r+1}=rg$$ boxes, and therefore it is completely filled by such a transposed tableau.  In other words, we are counting the number of `red' tableaux that precisely fill an $(r+1)\times rs$ rectangle where $s=\frac{g}{r+1}$.  Note that such tableaux exist if and only if $r+1$ divides $g$.

\begin{proposition}
    The number of transposed SSYT's of $(r+1)\times rs$ rectangle shape (where $s=\frac{g}{r+1}$) and content $(r^g)$ is equal to \begin{equation}\label{eq:Castelnuovo}g!\cdot \frac{1!\cdot 2!\cdots r!}{s!\cdot (s+1)!\cdot \cdot \cdots \cdot (s+r)!}.\end{equation}
\end{proposition}

\begin{proof}
These tableaux, under the bijection $\varphi$, correspond precisely to the standard Young tableaux of rectangle shape $(r+1)\times s$.  The classical `hook length formula' (see \cite[Ch.\ 7]{Stanley}) then results in the formula \eqref{eq:Castelnuovo}.
\end{proof}

\begin{example}
Suppose $r=4$ and $g=10$, so $d=r+\frac{rg}{r+1}=12$ and $s=\frac{g}{r+1}=2$.  Then one of the rectangular transposed tableaux enumerating $L_{g,r,d}$ is shown in red on the left below.  Its image under $\varphi$ is a $180^\circ$-rotated SYT filling an $(r+1)\times s$ box as shown in purple on the right below.
$$\begin{ytableau}
 \RED 2 & \RED 3 & \RED 5& \RED 6 & \RED 7 & \RED 8& \RED 9 &\RED 10\\
 \RED 1 & \RED 3 & \RED 4& \RED 5 & \RED 7 & \RED 8& \RED 9 &\RED 10 \\
 \RED 1 & \RED 2 & \RED 4& \RED 5 & \RED 6 & \RED 7& \RED 9 &\RED 10\\
 \RED 1 & \RED 2 & \RED 3& \RED 4 & \RED 6 & \RED 7& \RED 8 &\RED 10\\
 \RED 1 & \RED 2 & \RED 3& \RED 4 & \RED 5 & \RED 6& \RED 8 &\RED 9 \\ 
\end{ytableau}
\hspace{2cm}
{\color{violet}
\begin{ytableau}
 4 & 1 \\
 6 & 2 \\
 8 & 3 \\
 9 & 5 \\
 10 & 7
\end{ytableau}}$$
\end{example}

\section{$L'$-tableaux and enumeration by $2^g$}\label{sec:gdk}

We give a tableau interpretation of $L'_{g,d,k}$ in this section, starting from equation \eqref{eq:Lprime}, and show that these tableaux are enumerated by $2^g$ to prove Theorem \ref{thm:main2}.

\subsection{The $L'$-tableaux}\label{sec:Lprime-interp} Recall that equation \eqref{eq:Lprime} states that if $k+g\le 2d+1$ and $2\le k\le d$:
$$L'_{g,d,k}=\int_{\Gr(2,d+1)} \sigma_1^g \sigma_{k-1}\left(\sum_{i+j=2d-g-k-1}\sigma_i\sigma_j\right)-\int_{\Gr(2,d)}\sigma_1^g \sigma_{k-2}\left(\sum_{i+j=2d-g-k-2} \sigma_i\sigma_j\right).$$

We first give an interpretation of the left hand integral in the equation above.  Recall that a \textbf{standard Young tableau} of size $g$ is an SSYT of size $g$ in which the numbers $1,2,\ldots,g$ are each used exactly once.

\begin{definition}
A \textbf{positive $L'$-tableau with parameters $(g,d,k)$} is a way of filling a $2\times (d-1)$ grid with:
\begin{itemize}
    \item A standard Young tableau of size $g$ in the lower left corner (shaded red),
    \item A shading of the $k-1$ rightmost boxes in the top row (gray),
    \item A skew SSYT in two letters $0,1$ on the remaining squares (blue).
\end{itemize}
\end{definition}

By rearranging so that we think of the $\sigma_{k-1}$ as last in each product, and applying Corollary \ref{cor:iPieri}, we see that the positive term in equation \eqref{eq:Lprime} is equal to the number of positive $L'$-tableaux.

\begin{example}
  Here is an example of a positive $L'$-tableau with parameters $(3, 7, 4)$.
\begin{center}
\begin{ytableau}
 \RED 3 & \BLUE 0 & \BLUE 1 & \GRAY & \GRAY & \GRAY   \\
 \RED 1 & \RED 2 & \BLUE 0 & \BLUE 0 & \BLUE 1 & \BLUE 1 
\end{ytableau}
\end{center}
\end{example}

The second term in \eqref{eq:Lprime}, which we are subtracting, is similarly given by a set of smaller tableaux that we call \textit{negative} tableaux.

\begin{definition}
    A \textbf{negative $L'$-tableau with parameters $(g,d,k)$}
is a filling of a $2\times (d-2)$ grid with:
\begin{itemize}
    \item A standard Young tableau of size $g$ in the lower left corner (shaded red),
    \item A shading of the $k-2$ rightmost boxes in the top row (gray),
    \item A skew SSYT in two letters $0,1$ on the remaining squares (blue).
\end{itemize}
\end{definition}

\begin{example}
Here is an example of a negative $L'$-tableau with parameters $(3,7,4)$.
\begin{center}
\begin{ytableau}
 \RED 3 & \BLUE 0 & \BLUE 1 & \GRAY & \GRAY   \\
 \RED 1 & \RED 2 & \BLUE 0 & \BLUE 0 & \BLUE 1  
\end{ytableau}
\end{center}
\end{example}

Notice that there exist positive $L'$-tableaux if and only if $(k-1)+g\le 2(d-1)$, which is slightly stronger than the given condition $k+g\le 2d+1$. In particular if $k+g=2d$ or $k+g=2d+1$ we have $L'_{g,d,k}=0$, so we restrict our attention to the case that $k+g\le 2d-1$.  

\subsection{Enumeration by $2^g$}

We now prove Theorem \ref{thm:main2}.

\begin{definition}
    For fixed $g,d,k$, write $L'_{+}$ and $L'_{-}$ for the set of positive and negative $L'$ tableaux respectively of type $(g,d,k)$.  Also write $\psi:L'_{-}\to L'_{+}$ for the map that takes a negative tableau $T$ and adds a blue $1$ to the end of the bottom row and a gray box to the end of the top row. 
\end{definition}

Our above analysis shows that $$L'_{g,d,k}=|L'_+|-|L'_{-}|,$$ and we analyze this difference combinatorially.  The definitions above directly show that $\psi$ is a well-defined injective map, and so \begin{equation}\label{eq:Lprime}L'_{g,d,k}=|L'_{+}\setminus \psi(L'_{-})|.\end{equation}  The following proposition characterizes the image $\psi(L'_{-})$. 

\begin{proposition}
  A positive tableau $T$ is equal to $\psi(S)$ for some negative tableaux $S$ if and only if the bottom row of $T$ contains a blue $1$.
\end{proposition}

\begin{proof}
By the definition of $\psi$, any tableau in its image has a blue $1$ on the bottom right.  Conversely, if the bottom row of $T$ contains a blue $1$, then by semistandardness of the blue tableau, the bottom-rightmost entry is a blue $1$ as well, and removing the last column of $T$ yields a negative tableau $S$ for which $\psi(S)=T$.
\end{proof}

Applying this proposition and equation \eqref{eq:Lprime}, we obtain the following combinatorial interpretation of $L'_{g,d,k}$.

\begin{corollary}
The quantity $L'_{g,d,k}$ is equal to the number of positive $L'$-tableaux with parameters $(g,d,k)$ for which the bottom row contains no blue $1$ (and hence the only blue numbers in the bottom row are $0$'s).
\end{corollary}

For sufficiently large $d$, we can simplify this characterization even further. 

\begin{MainThm2}
If $d\ge g+k$, we have $L'_{g,d,k}=2^g$.
\end{MainThm2}

\begin{proof}
  Suppose $d\ge g+k$.  Then $d-1\ge g+(k-1)$, so the red and gray boxes of any positive $L'$-tableau with parameters $(g,d,k)$ cannot share a column.  In particular, for any positive tableau $T$ that has no blue $1$ in the bottom row, there are all blue $0$'s under the gray squares, and moreover any remaining columns to the right of the red tableau are uniquely determined (having one $0$ and one $1$) as well.  Thus the data determining $T$ is entirely contained in its first $g$ columns, which consists of a red standard tableau $Q$, and a binary tableau $P$ of the same shape as $Q$, where $P$ is obtained by rotating the blue numbers in these columns $180^\circ$ and replacing all $0$'s with $1$'s and $1$'s with $0$'s.
  
  By the RSK correspondence, these pairs $(P,Q)$ are precisely in bijection with the binary sequences of length $g$, and so we have that $L'_{g,d,k}=2^g$ as desired.
\end{proof}

\begin{example}
  The tableau below at left is a positive $L'$-tableau with parameters $(3,7,4)$ that is not the image of a negative one.  Since $g=3$, we restrict our attention to the first three columns (second image below), then consider the associated pair of tableaux of the same shape by rotating the blue tableaux and inverting the labels.  Finally, this pair corresponds under RSK to a unique length $3$ binary sequence. 
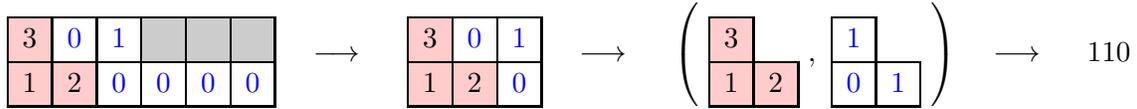
\begin{center}
\begin{ytableau}
 \RED 3 & \BLUE 0 & \BLUE 1 & \GRAY & \GRAY & \GRAY   \\
 \RED 1 & \RED 2 & \BLUE 0 & \BLUE 0 & \BLUE 0 & \BLUE 0 
\end{ytableau}\hspace{0.5cm}$\longrightarrow$\hspace{0.5cm}
\begin{ytableau}
 \RED 3 & \BLUE 0 & \BLUE 1   \\
 \RED 1 & \RED 2 & \BLUE 0  
\end{ytableau}\hspace{0.5cm}$\longrightarrow$\hspace{0.5cm}
$\left(\,\,\begin{ytableau}
 \RED 3   \\
 \RED 1 & \RED 2  
\end{ytableau}\,\,,\,\,
\begin{ytableau}
 \BLUE 1   \\
 \BLUE 0 & \BLUE 1  
\end{ytableau}\,\,\right)$\hspace{0.5cm}$\longrightarrow$\hspace{0.5cm}
110
\end{center}
\end{example}

\section{Further combinatorial observations}\label{sec:combinatorics}

In this section we provide two generalizations/observations regarding the combinatorics discussed above.  In particular we consider two variations of $L$-tableaux and explore their properties.

\subsection{Generalizing the map $\varphi$}
\label{sec:general}

We note that the `red to purple' bijection $\varphi$ may be generalized to transposed tableaux with $i$ of each entry (for any positive integer $i\le r$) as follows.

\begin{definition}
    Let $\mathrm{TrSSYT}(g,r,i)$ be the set of all transposed SSYT's of content $(i^g)=(i,i,\ldots,i)$ and height $\le r+1$.
\end{definition}

In particular, setting $i=r$ gives us the red tableaux defined in Definition \ref{def:red}, and setting $i=1$ gives us the set of standard Young tableaux of size $g$ and height $\le r+1$.

\begin{proposition}\label{prop:general}
There is a bijection $\varphi_i:\mathrm{TrSSYT}(g,r,i)\to \mathrm{TrSSYT}(g,r,r+1-i)$ for each $i$, that agrees with the bijection $\varphi$ of Definition \ref{def:varphi} at $i=r$ (up to a $180^\circ$ rotation of the output).
\end{proposition}

See Figure \ref{fig:height-i} for an example of this map for $i=3$ and $r=4$.

In fact, the map $\varphi_i$ can be realized as a special case of an even more general map studied by Stanley \cite{Stan83}.  It was later studied by Reiner and Shimozono \cite{REINER19981}, who call the map the \textit{box complement} and study it on a generalization of partition diagrams called \textit{\%-avoiding shapes}.  Proposition \ref{prop:general} may be proven using similar methods to our proof of Lemma \ref{lem:RedToPurple}, but we simply refer to \cite{REINER19981} for an existing proof in the more general setting.

\begin{remark}
For $i=r$, the transposed SSYT's of height at most $r+1$ are in bijection with standard Young tableaux, and therefore may be enumerated by the hook length formula for any given shape. It would be interesting to investigate whether there is a hook-length-like formula enumerating transposed SSYT's with content $(i^g)$ and height at most $r+1$ for $1<i<r$.
\end{remark}

\begin{figure}
    \centering
    \begin{ytableau}
     \RED 2 & \RED 5  \\
     \RED 2 & \RED 4 & \RED 5 \\
     \RED 1 & \RED 3 & \RED 5 \\
     \RED 1 & \RED 3 & \RED 4\\
     \RED 1 & \RED 2 & \RED 3 & \RED 4\\
    \end{ytableau}
    \hspace{2cm}
    {\color{violet}
    \begin{ytableau}
    4 & 3 & 1 \\
    \none  & 3 & 1 \\
     \none & 4 & 2 \\
    \none & 5 & 2 \\
    \none & \none & 5 
    \end{ytableau}
    }
    \caption{A tableau in $\TrSSYT(5,4,3)$ and its image under $\varphi_3$ in $\TrSSYT(5,4,2)$.}
    \label{fig:height-i}
\end{figure}

\subsection{Restricting the alphabet}

We now ask whether our bijective constructions can lead to related interesting enumeration problems.  In particular, one natural variant we may consider is limiting the alphabet of the blue tableau (in the $L$-tableau setting) to a smaller size, so that under RSK we end up with words in a smaller alphabet.

\begin{definition}
	Define a \textbf{restricted $L$-tableau with parameters $(g,r,g+r,i)$} to be an $L$-tableau of paramters $(g,r,g+r)$ where we restrict the alphabet of the blue integers to $\{0,1,\ldots,r-i\}$.
\end{definition} 
Note that the parameter $g+r$ is redundant, and we simply include it for consistency with the parameter $d$ in our previous notation.  For larger $d$ there would be no restricted $L$-tableaux, because a full column of height $r+1$ cannot be filled by blue integers from $\{0,1,\ldots,r-i\}$ in a semistandard tableau.

It turns out that this restricted setting simply reduces to a smaller case of our usual $L$-tableaux.

\begin{proposition}
	The number of restricted $L$-tableaux with parameters $(g,r,g+r,i)$ is equal to the number of $L$-tableaux with parameters $(g,r-i,g+r)$.
\end{proposition}

\begin{proof}
	By Lemma \ref{lem:chop} we may assume we are working with truncated tableaux. 
	Since the blue integers are restricted to the alphabet $\{0,1,\ldots,r-i\}$, the blue tableau has height at most $r-i$. So, the bottom $i$ rows of the red tableau have width $g$. These rows must be filled with each integer $1,\ldots,g$. As there is a unique way to do this, the act of removing the bottom $i$ rows of the $(r+1)\times(g)$ grid gives a bijection between the restricted $L$-tableaux with parameters $(g,r,d,i)$ and the $L$-tableaux with parameters $(g,r-i,d)$.
\end{proof}

Combining this proposition with Theorem \ref{thm:main1} yields the following enumerative result.

\begin{corollary}
There are exactly $(r-i+1)^g$ restricted $L$-tableaux with parameters $(g,r,g+r,i)$.
\end{corollary}

\bibliography{refs}
\bibliographystyle{plain}

\end{document}